\documentclass[11pt]{article}

\usepackage{amsmath,amsthm,amssymb,color}
\usepackage{graphicx}
\usepackage{tikz-cd}

\parskip=\medskipamount
\def\fpd#1#2{{\displaystyle\frac{\partial #1}{\partial #2}}}

\def\clift#1{#1^{\scriptscriptstyle{\mathrm{C}}}}
\def\hlift#1{#1^{\scriptscriptstyle{\mathrm{H}}}}
\def\vlift#1{#1^{\scriptscriptstyle{\mathrm{V}}}}

\def\R{\mathbb{R}}

\def\onehalf{{\textstyle\frac12}}

\def\oneq{{\textstyle\frac14}}
\def\threeq{{\textstyle\frac34}}

\font\frak=eufm10 scaled\magstep1
\def\goth#1{\hbox{{\frak #1}}}
\def\g{\goth{g}}
\def\la{\g}

\def\cinfty#1{C^{\scriptscriptstyle\infty}(#1)}
\def\vectorfields#1{\goth{X}(#1)}

\def\sode{{\sc sode}}
\def\sodes{{\sc sode}s}

\def\T{{\mathbf T}}

\theoremstyle{plain}

\newtheorem{proposition}{Proposition}
\newtheorem{definition}{Definition}

\begin{document}

\title{Conjugate points for systems of second-order ordinary differential equations}

\author{ S.\ Hajd\'u and T.\ Mestdag  \\[2mm]
{\small Department of Mathematics,  University of Antwerp,}\\
{\small Middelheimlaan 1, 2020 Antwerpen, Belgium}\\
{\small Email: sandor.hajdu@uantwerpen.be, tom.mestdag@uantwerpen.be} 
}

\date{}

\maketitle

\begin{abstract}
We recall the notion of Jacobi fields, as it was  extended to systems of second-order ordinary differential equations. Two points along a base integral curve are conjugate if there exists a non-trivial Jacobi field along that curve that vanishes on both points. Based on arguments that involve the eigendistributions of the Jacobi endomorphism, we discuss conjugate points for a certain generalization (to the current setting) of locally symmetric spaces. Next, we study conjugate points along relative equilibria of Lagrangian systems with a symmetry Lie group. We end the paper with some examples and applications.
\vspace{3mm}


\vspace{3mm}

%\footnote{	Symmetries, Lie-group and Lie-algebra methods in mechanics..}
\end{abstract}

\section{Introduction}

Beyond question Jacobi fields play a fundamental role in Riemannian geometry \cite{DoCarmo}, and also in Finsler geometry \cite{BCS2} they have proven their extreme usefulness. The geodesic spray of a Riemannian or Finsler metric is but one of many examples of a so-called `semispray', or `\sode\ vector field'. These are vector fields on a tangent manifold whose integral curves can be associated to a system of second-order ordinary differential equations. 
 It is probably less known that the concept of Jacobi fields has  been extended to the context of \sodes\ (see e.g.\ \cite{CM,CGM,Foulon,Geoff,Bucataru}). 
The idea remains the same: a Jacobi field measures the infinitesimal variation of a 1-parameter family of solution curves of the \sode.   
The main observation to keep in mind is that, in the generalized Jacobi equation at hand,  the Levi-Civita connection and the curvature of the Riemannian metric are replaced by the (more general) covariant dynamical derivative $\nabla$ and the Jacobi endomorphism $\Phi$ of the \sode.

After recalling the calculus of tensor fields along the tangent bundle projection, and after introducing the above mentioned $\nabla$ and $\Phi$, we recall the notion of a Jacobi field in section~\ref{sec3}. We will show that sprays can be understood as those \sodes\ which posses the trivial Jacobi fields $\dot c$ and $t\dot c$ (Proposition~\ref{Prop3}). 

We then turn our attention to the concept of conjugate points. Given an initial point of a base integral curve  of a \sode, these are points (further along the same curve) which have the property that there exists a Jacobi field that vanishes on both the initial point and there. We show that this notion can be generalized to  \sodes, and relate it in Proposition~\ref{Prop6} to the critical points of the exponential map (as it was introduced for \sodes\ in \cite{MDM}). 

Since in practical applications it is not always possible to compute the exponential map in an explicit form, we need to devise other techniques to find conjugate points. In Proposition~\ref{basicprop} we give a very general method, based on the availability of a constant and positive eigenfunction of the Jacobi endomorphism, and of a parallel vector field along a given base integral curve of the \sode. We then apply the method in two different situations.

Both of these applications will require \sodes\ that satisfy $[\nabla\Phi,\Phi] =0$ (or, satisfy this property when restricted to some eigendistribution of $\Phi$). This condition is quite familiar in the context of \sodes. For example, it is one of the conditions for a \sode\ to belong to ``Case II'' of the ``Inverse problem of Lagrangian mechanics'' (see e.g.\ \cite{CPST}), and it is one of many conditions for a \sode\ to be ``separable'' (see e.g.\ \cite{sep}).
First, in Section~\ref{symspace}, we discuss conjugate points for a certain generalization to the current setting of locally symmetric Riemannian spaces (Proposition~\ref{Prop11}). In this situation the eigenfunctions are first integrals of the \sode. Next, in Section~\ref{sec6}, we study conjugate points for \sodes\ with a symmetry Lie group. We show that in this case the eigenfunctions of the Jacobi endomorphism are invariant functions (Proposition~\ref{Prop13}). They remain, for this reason, constant along relative equilibria. This enables us to give sufficient conditions for the existence of conjugate points along a relative equilibrium (Proposition~\ref{jacobiRE}).

We end the paper with some examples and applications to surfaces of revolution, the free rigid body and the canonical connection on a Lie group. In each of these examples, we  link our results to those in the literature. 

\section{Calculus along the tangent bundle projection}

Let $M$ be a manifold. We start with a short survey of the so-called calculus of tensor fields
along the tangent bundle projection $\tau: TM\to M$, as introduced in \cite{MaCaSaI,MaCaSaII}. For a short exposition, see e.g.\ \cite{Sarlet}.

A {\em vector field along $\tau$} is a  section of the pullback bundle
$\tau^*TM \to TM$. We will write $\vectorfields \tau$ for ${\rm Sec}(\tau^*TM)$, from now on. Such a section can also be thought of as a map $X: TM \to TM$ with
the property that $\tau\circ X = \tau$. Any vector field $Y$ on $M$ induces a (so-called) `basic' vector field $X=Y\circ\tau$ along $\tau$. In what follows, we will often simply write $Y$, even when we mean its interpretation $Y\circ\tau$ as a vector field along $\tau$.

Let $(q^i)$ be local coordinates on $M$, and $(q^i,{\dot q}^i)$ be its induced natural coordinates on $TM$. In general, a vector field $X$ along $\tau$ can locally be expressed as
\[
X=X^i(q,\dot q)\fpd{}{q^i} \in\vectorfields\tau,
\]
where $\fpd{}{q^i}$ are the coordinate vector fields on $M$, in their intepretation as vector fields along $\tau$. For example, we may always view the identity $\dot q\mapsto \dot q$ in a canonical way as a vector field along $\tau$. If we denote the correspoding section as $\T$, then
\[
\T={\dot q}^i \fpd{}{q^i}.
\]

Vector fields along $\tau$ are in 1-1 correspondence with vertical vector fields on $TM$: each $X\in\vectorfields\tau$ can be vertically lifted to  $\vlift{X} \in\vectorfields{TM}$, given by
\[
\vlift X= X^i\fpd{}{{\dot q}^i}.
\]
In particular, $\vlift{\T}=\Delta$, the Liouville vector field.

In similar fashion as for vector fields along $\tau$, we will speak below of {\em tensor fields along $\tau$}. 

 If $Y=Y^i(q)\partial/\partial q^i$ is a vector field on
$M$, its complete lift $\clift{Y}$ is the following vector field on
$TM$:
\[
\clift Y= Y^i \fpd{}{q^i} + \fpd{Y^i}{q^j}{\dot q}^j\fpd{}{{\dot
q}^i}.
\]
The relations between the brackets of complete and
vertical lifts of vector fields $Y_1$ and $Y_2$ on $M$ are:
\[
[{\clift Y_1},{\clift Y_2}] = [Y_1,Y_2\clift], \qquad [{\clift Y_1},{\vlift Y_2}] =
[Y_1,Y_2\vlift] \qquad\mbox{and}\qquad [{\vlift Y_1},{\vlift Y_2}] = 0.
\]

A {\em second-order differential
equation field} $\Gamma$ (from now on \sode, in short) is a vector field on $TM$ with the property that all its integral curves are lifted curves $\dot c(t)$ of curves $c(t)$ in $M$ (the so-called base integral curves of $\Gamma$). A \sode\ is  locally given by
\[
\Gamma={\dot q}^i\fpd{}{q^i}+f^i\fpd{}{{\dot q}^i}.
\]

It  can be
used to define the horizontal lift $\hlift{X} \in \vectorfields{TM}$ of
$X\in\vectorfields\tau$:
\[
\hlift X= X^i \left(\fpd{}{q^i} - \Gamma^j_i \fpd{}{{\dot
q}^j}\right), \qquad \Gamma^j_i =-\onehalf \fpd{f^j}{{\dot q}^i}.
\]
Any vector field $Z$ on $TM$ can then be decomposed into a horizontal and
vertical component:\ $Z={\hlift X_1} + {\vlift X_2}$, for  $X_1,X_2 \in
\vectorfields\tau$. In case $Y$ is a vector field on $M$,
the three lifts are related as follows:
\[
\hlift Y =\onehalf (\clift Y - [\Gamma,\vlift Y]).
\]

The properties of a \sode\ $\Gamma$ that are of interest to us can often be derived from an analysis of its {\em Jacobi endomorphism} $\Phi$ and its {\em dynamical covariant derivative} $\nabla$. These two important concepts can be defined  by considering the Lie bracket of $\Gamma$ with either  horizontal or vertical lifts. For $X\in\vectorfields{\tau}$, these brackets take the form
\[
[\Gamma,\vlift X] = -\hlift X + (\nabla X\vlift)
\qquad\mbox{and}\qquad [\Gamma,\hlift X]= (\nabla X\hlift) +
(\Phi(X)\vlift).
\]
The operator $\Phi$ is a type (1,1) tensor field along $\tau$.  The operator $\nabla$, on the other hand, acts
as a derivative on $\vectorfields\tau$, in the sense that for
$f\in\cinfty{TM}$ and $X\in\vectorfields\tau$, 
\[
\nabla(fX) = f\nabla X + \Gamma(f)X.
\] 

The coordinate  expressions for  $\nabla$ and $\Phi$ are
\[
\nabla \fpd{}{q^j} = \Gamma_j^i \fpd{}{q^i}, \qquad \Phi\left( \fpd{}{q^j}\right)= \Phi^i_j \fpd{}{q^i} = \left(-\fpd{f^i}{q^j} - \Gamma^k_j\Gamma^i_k - \Gamma(\Gamma^i_j)\right)\fpd{}{q^i}.
\]

In what follows it will be advantageous to distinguish between the concepts as above introduced, and their restrictions to a specific (lifted) curve in $TM$. A vector field  along a curve $c$ is a map
$W:\R\to TM$ with $\tau(W(t)) = c(t)$. We will denote the set of such vector fields by $\vectorfields c$.

For a given base integral curve $c$ of $\Gamma$ and a vector field $X\in\vectorfields\tau$, we may define such a vector field $X_c$ along $c$ by means of the map $X_c:t\in\R\mapsto X({\dot c}(t))$, since by definition $\tau(X({\dot c}(t))) = \tau({\dot c}(t)) = c(t)$. 

For any $v\in T_mM$, we may consider the endomorphism $\Phi_v: T_mM \to T_mM$. The collection of those for $v={\dot c}(t)$ can be interpreted as an operator $\Phi_c$ that maps vector fields along $c$ to vector fields along $c$. When $W(t)=W^i(t) \fpd{}{q^i}\bigg\rvert_{c(t)}\in\vectorfields c$ then
\[
\Phi_c(W(t)) = \Phi^i_j({\dot c}(t)) W^j(t) \fpd{}{q^i}\bigg\rvert_{c(t)} \in\vectorfields c.
\] 
Likewise, by the relation
\[
\nabla_c W (t) = \left(\frac{d}{dt}W^i(t) + \Gamma^i_j({\dot c}(t)) W^j(t) \right)\fpd{}{q^i}\bigg\rvert_{c(t)}
\]
we define an operator $\nabla_c: \vectorfields{c} \to \vectorfields{c}$ with the property
\[
\nabla_c(\mu(t)W(t)) = \dot \mu(t) W(t) + \mu(t) \nabla_c W(t), \qquad \mu \in \cinfty{\R}.
\]

With these definitions and notations, it is clear that for any $X\in\vectorfields \tau$, 
\[
 \Phi(X)({\dot c}(t)) = \Phi_c(X_c(t))\qquad \mbox{and}\qquad  \nabla X({\dot c}(t)) = \nabla_c X_c(t).
\]

These two concepts are extensions (to the current setting) of more familiar objects in e.g.\ Riemannian geometry. When $g$ is a Riemannian metric, $\Gamma$ its corresponding geodesic (quadratic) spray and $D$ its Levi-Civita connection then
\[
\Phi_c (X_c(t)) = R(X_c(t),{\dot c(t)}){\dot c}(t) \qquad \mbox{and} \qquad \nabla_c X_c(t) = D_{{\dot c}} X ({\dot c}(t)), 
\]
where $R(Y_1,Y_2)Y_3=D_{Y_1}D_{Y_2} Y_3 - D_{Y_2}D_{Y_1} Y_3 - D_{[Y_1,Y_2]} Y_3$ stands for the curvature of $D$.

\section{Jacobi fields and conjugate points for \sodes} \label{sec3}

In the paper \cite{CM} (see also \cite{CGM}) the notion of a Jacobi field  has been extended to the context of  \sodes. It is based on the notion of a variational vector field.

\begin{definition}
	A 1-parameter family of integral curves of a vector field $Y\in\vectorfields{M}$ is a map $\zeta: ]-\epsilon,\epsilon[\times I\subset \R^2\to M$ such that for every $s\in]-\epsilon,\epsilon[$ the curve $\zeta_s:I\to M$, given by $\zeta_s(t):=\zeta(s,t)$ is an integral curve of $Y$. The vector field $Z$ along $\zeta_0$ defined by $Z(t)=\frac{\partial\zeta}{\partial s}(0,t)$ is said to be the variation vector field defined by the 1-parameter family.
\end{definition}

We follow \cite{CGM} to give an infinitesimal characterization of such a vector field: 
	A vector field $Z$ along an integral curve $\zeta_0$ of $Y$ is the infinitesimal variation defined by a 1-parameter family of integral curves of $Y$ if and only if $\mathcal{L}_YZ(t)=0$ for all $t\in I$. For more equivalent conditions we refer to Proposition 2.3 of \cite{CGM}.

We now consider the case, when the variational vector field $Z(t)$ is constructed w.r.t a \sode\ $\Gamma$ on $TM$ along one of its integral curves $\zeta_0$. In this case  - taking into account that the integral curves of $\Gamma$ are all lifted curves - the variation $\zeta(s,t)$ can be written as $\zeta(s,t)=\frac{\partial\gamma}{\partial t}(s,t)$, where $\gamma(s,t)$ is a 1-parameter family of base integral curves of the \sode. If we denote by $W(t)$ the variational vector field of the base family, that is 
\[
W(t) = \frac{\partial \gamma}{\partial s}(0,t),
\]
then $Z(t)=W^c(t)$. We are now ready to define a (generalized) Jacobi field.

\begin{definition}
	Let $c$ be a base integral curve of a \sode\ $\Gamma$. A Jacobi field along $c$ is a vector field $J(t)$ along $c$, whose complete lift is a variational vector field along the integral curve $\dot c$ by integral curves of $\Gamma$.
\end{definition}
According to Theorem 2.7 of \cite{CGM}, a vector field along $c$ is a Jacobi field if and only if it satisfies the (generalized) Jacobi equation
\[
\nabla_c\nabla_c J(t) + \Phi_c(J(t)) = 0.
\]

It is clear that if $X\in \vectorfields{\tau}$ is such that $\nabla\nabla X +\Phi(X)=0$, then $J(t)=X_c(t)=X({\dot c}(t))$ is a Jacobi field for any choice of $c$.

When below we refer to `the Riemannian case', we mean the situation where $\Gamma$ is the (quadratic) geodesic spray of a Riemannian metric. We remark that, in the case when in addition $I$ is a compact interval (or the manifold is geodesically complete), the variational vector fields along a geodesic through geodesics are in 1-1 correspondence with Jacobi fields. (See  Theorems 10.1  and  10.4 of \cite{JohnRiem}). 

We also know that  in that case $\dot c$ and $t\dot c$ are always Jacobi fields (see e.g.\ \cite{DoCarmo}). For arbitrary \sodes, however, this will not always be the case, as we now show. 

A \sode\ $\Gamma$ is said to be a {\em spray} if $[\Delta,\Gamma]=\Gamma$, where $\Delta=\vlift\T$ is the Liouville vector field. A  coordinate calculation will easily confirm that $\Gamma$ is a spray if and only if $\Gamma=\hlift\T$.

\begin{proposition} \label{Prop2}
A \sode\ $\Gamma$ is a spray if and only if $\nabla\T = 0$ and $\Phi(\T) = 0$.
\end{proposition}

\begin{proof}
When $\Gamma$ is a spray, then  \[
-\Gamma = [\Gamma,\vlift \T] = -\hlift \T + (\nabla \T \vlift)
\qquad\mbox{and}\qquad 0 = [\Gamma,\hlift \T]= (\nabla \T\hlift) +
(\Phi(\T)\vlift).
\]
From the first relation we obtain $\nabla\T = 0$ and then, from the second, $\Phi(\T) = 0$.

Conversely, when both $\nabla\T = 0$ and $\Phi(\T) = 0$, then $0 = [\Gamma,\hlift \T]$. Since the difference between two \sodes\ is always vertical, we may write $\T^H=\Gamma +\vlift X$. But then also $0 = [\Gamma,\vlift X] = -\hlift X + (\nabla X \vlift)$. Therefore $\hlift X=0$, and thus $X=0$. We conclude that $\Gamma=\hlift\T$.\end{proof}

\begin{proposition} \label{Prop3}
The vector fields $\dot c$ and $t\dot c$ along $c$ are both Jacobi fields for each base integral curve $c$ of a \sode\ $\Gamma$ if and only if  $\Gamma$ is a spray.
\end{proposition}

\begin{proof}
If a \sode\ is a spray then  $\nabla \T=0$ and $\Phi(\T)=0$. Since $\T_{c} = \dot c$, also $\nabla_c \dot c = 0$ and $\Phi_c(\dot c) =0$. Thus,  $J(t)=\dot c(t)$ is a Jacobi field.

Let now $J(t)=t \dot c$ (there is no corresponding $X\in\vectorfields \tau$ for which this is $X_c$), then $\nabla_c J(t) = \dot c $ and $\nabla_c \nabla_c J(t) = \nabla_c \dot c(t) = 0$. Moreover, also $\Phi_c (J(t)) =t\Phi_c(\dot c(t))=0$.

Conversely, if both $\dot c$ and $t\dot c$ are Jacobi fields, then $\nabla_c\nabla_c\dot c + \Phi_c(\dot c) = 0$ and 
\[
0=\nabla_c\nabla_c(t\dot c) + t \Phi_c(\dot c) =    \nabla_c (t\nabla_c(\dot c)) + \nabla_c \dot c + t \Phi_c(\dot c)   =    t\nabla_c \nabla_c(\dot c) + \nabla_c\dot c + \nabla_c \dot c + t \Phi_c(\dot c)   = 2 \nabla_c\dot c.
\] 
Since this holds for each $c$, we get $\nabla \T = 0$. Moreover, since $\dot c$ is a Jacobi field, we obtain from $\nabla_c\dot c = 0$  and the Jacobi equation that $\Phi_c(\dot c) = 0$, which leads to $\Phi(\T) = 0$.
\end{proof}

Besides Riemannian metrics, also Finsler metrics have the property that their geodesic equations are governed by a spray. As we will recall in the examples below, there exist, however, sprays whose base integral curves can never be the geodesics of a Riemannian  or a Finsler metric.

Consider a frame $\{e_i(t)\}$ along $c$. %(we do not assume $\nabla_c e_i(t)=0$, e.g.\ take $e_i(t) = \fpd{}{x^i}\mid_{c(t)}$). 
Any vector field along $c$ can be written as $J(t) =J^i(t)e_i(t)$. This vector field will be a Jacobi field if it satisfies an equation of the type
\[
{\ddot J}^i(t) + A^i_j(t){\dot J}^j(t)+ B^i_j(t){J}^j(t)=0,
\] 
i.e.\ a linear equation. This equation is locally determined by the knowledge of an initial value and an initial velocity. In particular the zero-solution is the only solution with zero initial value and zero initial velocity, and a linear combination of  solutions is again a solution. 

When $J_1(t)$ and $J_2(t)$ are two Jacobi fields, then so is also $aJ_1(t)+bJ_2(t)$. We show that they are linearly independent if and only if the $2n$-vectors $(J^i_1(0),{\dot J}^i_1(0))$ and $(J^i_2(0),{\dot J}^i_2(0))$ in $\R^{2n}$ (containing the initial data) are linearly independent. Indeed, suppose that  $aJ_1(t)+bJ_2(t) = 0$. Since it is a solution it must be the zero-solution. But then,  the initial values $(aJ^i_1(0)+bJ^i_2(0),a{\dot J}^i_1(0)+b{\dot J}^i_2(0))$ are just $(0,0)$. Due to the assumption on linear independence we get that $a=b=0$.

We may therefore conclude:
\begin{proposition} \label{prop4}
For any \sode\ there exist at most $2n$ linearly independent Jacobi fields along each base integral curve.
\end{proposition}

We can now extend the following definition of e.g.\ \cite{DoCarmo} to the current context of \sodes.

\begin{definition}
Let $c$ be a base integral curve of a \sode\ $\Gamma$, through $m_0=c(0)$. If there exists a Jacobi field $J(t)$, not identically zero, with the property that $J(0) = J(t_1)=0$, then the point $m_1=c(t_1)$ is called a conjugate point of $m_0$ along $c$. 
\end{definition}

The maximum number of such linearly independent fields is called the {\em multiplicity} of the conjugate point. 
If we fix $J(0)=0$ then, in view of Proposition~\ref{prop4}, there exists at most $n$ linearly independent Jacobi fields (determined by the linear independence of their initial velocity). 

We have shown in Proposition~\ref{Prop3} that, when $\Gamma$ is a spray (as it will be in the case when $\Gamma$ is the geodesic spray of a Riemannian or Finslerian metric),  $J(t) = t\dot c(t)$ is always a Jacobi field. Since it never  vanishes  at any $t\neq 0$ (meaning that it can not be used to give a conjugate point) we may conclude that   for a spray the multiplicity of a conjugate point is {\em at most} $n-1$ (see  \cite{DoCarmo} for this statement in the Riemannian case).

In \cite{other,MDM} the notion of the exponential map is extended to the context of \sodes. We will use this map to characterize conjugate points. The construction is based on the following proposition:

\begin{proposition} \label{Prop5} {\rm \cite{MDM}}
	Let $\Gamma$ be a \sode\ on $M$. Then, for every point $m_0\in M$ there exists a sufficiently small positive number $t_1$ and two open subsets $U,\tilde{U}$ of $M$ with $m_0\in U\subseteq \tilde{U}\subseteq M$, such that for all $m_1\in U$ there exists a unique solution of $\Gamma$, 
	\[
	c_{m_0m_1}:[0,t_1]\to\tilde{U},
	\]
	satisfying
	\[
	c_{m_0m_1}(0)=m_0,\qquad c_{m_0m_1}(t_1)=m_1.
	\]
\end{proposition}

Let us denote by $\varphi^\Gamma:D^{\Gamma}\subseteq \R\times TM\to TM$ the flow of $\Gamma$, where $D^{\Gamma}$ is the open subset of $\R\times TM$ given by
\[
D^{\Gamma}=\{ (t,v)\in\R\times TM \phantom{a}\rvert\phantom{a} \varphi^\Gamma(\cdot,v)\phantom{a}\text{is defined at least in}\phantom{a} [0,t]\}.
\]
Now for any $t_1\geq 0$ we can define the open subset $D^{\Gamma}_{(t_1,m_0)}$ of $T_{m_0}M$ by
\[
D^{\Gamma}_{(t_1,m_0)}=\{v\in T_{m_0}M \phantom{a}\rvert\phantom{a} (t_1,v)\in D^{\Gamma}\}.
\]
For sufficiently small $t_1$, this set is non-empty. Finally, let $\mathcal{U}\subseteq D^{\Gamma}_{(t_1,m_0)}$ be the open subset of $T_{m_0}M$ given by
\[
\mathcal{U}=\{\dot{c}_{m_0m_1}(0)\subseteq D^{\Gamma}_{(t_1,m_0)}\phantom{a}\rvert\phantom{a}m_1\in U \}.
\]
We can now define the exponential mapping at the point $m_0$ for time $t_1$ by
\begin{eqnarray*}
	\exp^{\Gamma}_{(t_1,m_0)}:\mathcal{U}\subseteq T_{m_0}M\rightarrow M, \\
	\exp^{\Gamma}_{(t_1,m_0)}(v)=(\tau_M\circ\varphi^{\Gamma})(t_1,v).
\end{eqnarray*}
We remark that, when $t_1$ is sufficiently small in the sense of Proposition~\ref{Prop5},  we can also define $\exp^{\Gamma}_{(t,m_0)}$ for all $t\in [0,t_1]$, and in \cite{other} (Proposition~2.3) it is shown that the domain of the exponential map increases when the parameter $t_1$ decreases. In particular $\exp^{\Gamma}_{(0,m_0)}$ is a constant mapping. It is shown in \cite{MDM}, that the exponential map is a diffeomorphism for all times $t_1$. However, we may also consider a bigger domain $\hat{U}_{(t_1,m_0)}$ for the exponential map at the point $m_0$ for time $t_1$ defined by
\[
\hat{U}_{(t_1,m_0)}:=\left\{v\in T_{m_0}M\bigg\rvert\,\, \text{the unique solution with}\,\, c(0)=m_0,\,\, \dot{c}(0)=v\,\, \text{is defined for time}\,\, t_1  \right\},
\]
on which the exponential map is still differentiable, but not necessarily bijective  (see e.g. \cite{other} for more details). We will now relate this map to Jacobi fields. 

Consider the following 1-parameter family of base integral curves of $\Gamma$:
\[
	\gamma : S\times [0,t_1]\rightarrow M,\phantom{a} \gamma(s,t):=\exp^{\Gamma}_{(t,m_0)}(v+sw),
\]
where $v,w\in \mathcal{U}$ and $S\subseteq\R$ is such that the right-hand side of the above is well-defined. Then, $\gamma(s,t)$ is a variation of base integral curves of $\Gamma$, and $J(t)=\frac{\partial \gamma}{\partial s}\rvert_{s=0}$ is a Jacobi field satisfying $J(0)=0$ and $\nabla_c J(0)=w$. We can always construct such a family, given the initial data: a point on the manifold, a base integral curve, and tangent vectors $v,w\in\mathcal{U}$.

When a Jacobi field $J(t)$ along a base integral curve $c$ is given 'a priori', we can  construct a variation of $c$ defined by a 1-parameter family of base integral curves by setting $v:=J(0)$ and $w:=\nabla_c J(0)$. Its variational vector field, in the sense of the previous paragraph,  gives back $J(t)$.  

Note, that not all 1-parameter families of solutions can be given in the form of the above $\gamma$, since solutions of a \sode\ are, in general, not invariant under a parameter transformation. We have therefore the following diagram:
\begin{center}
	\begin{tikzcd}
	\{\text{general variation} \}
	\arrow[r, description] & J(t)\arrow[r] 
	%\arrow[l, bend left, dotted, "//" marking] 
	& \{\gamma \phantom{a}\text{in the above form}\}\arrow[l, bend left]   \\
	\end{tikzcd}
\end{center}

We are now able to characterize conjugate points with the help of the exponential map. In the statement, we use the vertical lift  $w^{\textsc{v}}_v \in T_vTM$  of $w\in T_mM$ to $v\in T_mM$. It is defined by
\[
w^{\textsc{v}}_v(f) = \frac{d}{dt} \bigg\rvert_{t=0} \Big(f(v+tw)\Big).
\] 

\begin{proposition} \label{Prop6}
Let $c$ be a base integral curve of $\Gamma$ with $\dot{c}(0)=v$ that joins $m_0=c(0)$ with $m_1=c(t_1)$ in $U$. Then, $m_0$ and $m_1$ are conjugate points if and only if there exists a tangent vector $w\in\mathcal{U}$ such that $w^{\textsc{v}}_v$ is a critical point of the exponential map at $v$.
\end{proposition}
\begin{proof}
We know that there exists a unique Jacobi field with $J(0)=0$ and $\nabla_c J(0)=w$. We show that this field vanishes at $t_1$ if and only if $w^{\textsc{v}}_v$ is a critical point of the exponential map at $v$. Indeed, by the chain rule we have
	\[
	J(t_1)=\frac{\partial}{\partial s}\bigg\rvert_{s=0}\exp^{\Gamma}_{(t_1,m_0)}(v+sw)= T_{v}\exp^{\Gamma}_{(t_1,m_0)}(w^{\textsc{v}}_v).
	\]
	\end{proof}

We will give an example of this proposition, when we discuss the canonical connection on a Lie group (Section~\ref{seccancon}).
The above proposition is, however, mainly theoretical in nature. In practice it is often difficult to find an explicit expression for the exponential map, for a given \sode. 
It is therefore of interest to construct some other methods to find conjugate points.

\section{A method  to find conjugate points}

In this section, we need the concept of a  distribution on the pullback bundle $\tau^*TM$. With a {\em $d$-dimensional distribution $D$ along $\tau$} we mean a smooth choice of a $d$-dimensional subspace $D(v)$ of $T_{\tau(v)} M$ for every $v\in TM$. We say that a
vector field $X$ along $\tau$ belongs to $D$ (and write  $X \in D$)  if $X(v) \in D(v)$ for each $v\in TM$. See e.g.\  \cite{sep} for more details on this concept.

Let $\Gamma$ be a \sode. Its Jacobi endomorphism $\Phi$ (as a tensor field along $\tau$) is said to be diagonalizable if for each $v\in TM$ the
endomorphism $\Phi_v:T_{\tau(v)} M \to T_{\tau(v)} M$ is diagonalizable, there exist (locally)
smooth functions $\lambda: TM \to \R$  (called eigenfunctions) such that $\lambda(v)$ is an eigenvalue of
$\Phi_v$ and the rank of $\lambda{\rm Id}-\Phi$ is constant. In this case, the eigenspaces of $\Phi$ define
distributions along $\tau$, called the eigendistributions of $\Phi$ and denoted by $D_{\lambda}$,
i.e.\ $D_\lambda = {\rm ker}(\lambda{\rm Id}-\Phi)$.

From now on, we will always assume that the Jacobi endomorphism is  diagonalizable. 
For a given base integral curve $c$ of $\Gamma$, we may consider $\Phi_c$ and the restrictions to a specific given base integral curve $c$ will be denoted by $\lambda_c(t) :=\lambda(\dot c(t))$ and $D_{\lambda_c} := \cup_t D_{\lambda_c(t)}$.

Remark that, when $\Gamma$ is a spray, we know from Proposition~\ref{Prop2} that $\Phi(\T)=0$. In case of a spray, we therefore always have the constant eigenvalue $\lambda=0$.

The following observation lies at the basis of most of what follows:
\begin{proposition} \label{basicprop}
Let $c$ be a base integral curve of a \sode\ $\Gamma$, through $m_0=c(0)$.  If 
\begin{enumerate}
\item[(1)] $\Phi$ has an eigenfunction $\lambda(v)$ that remains  constant and strictly positive  along $c$, i.e.\ $\lambda_c(t)=\lambda_0>0$ for all $t$,
\item[(2)]  there exists a non-vanishing  vector field  $V(t) \in \vectorfields c$ along $c$ that lies in $D_{\lambda_c}$, and which is such that $\nabla_c V(t) = 0$, 
\end{enumerate}
then the  points $c\big(\frac{k\pi}{\sqrt{\lambda_0}}\big)$  are conjugate to $m_0$, for all $k\in{\mathbb Z}$.
\end{proposition}
\begin{proof}
We show that, under the above assumptions, the Jacobi equation has a solution of the type $J(t)=\varpi(t)V(t)$. 
Since
\[
\nabla_c(\nabla_c J(t)) = \nabla_c ( \dot\varpi(t)V(t) + \varpi(t)\nabla_c V(t))= \ddot \varpi(t)V(t) + \dot\varpi(t)\nabla_c V(t) = \ddot \varpi(t)V(t),
\]
we find after substitution  in the Jacobi equation  (given that $V(t)\neq 0$):
\[
{\ddot \varpi}(t) +\lambda_0 \varpi(t)=0.
\] 
Since $\lambda_0 >0$, the solutions with  $\varpi(0)=0$ are given by $\varpi(t) = A\sin(\sqrt{\lambda_0}t)$. In that case $J\big(\frac{k\pi}{\sqrt\lambda_0}\big)=0$, for all $k\in {\mathbb Z}$.
\end{proof}

Remark that the assumption (2) is not automatically satisfied. Let $c$ be a base integral curve of $\Gamma$, through $c(0)=m_0$ and with ${\dot c}(0)=v\in T_{m_0}M$. Given any other vector $w \in T_{m_0}M$ we may define a vector field $V(t)$ along $c$ with initial value $w$ by demanding that $\nabla V(t)=0$. Indeed, this relation is just the initial value problem
\[
{\dot V}^i(t) + \Gamma^i_j({\dot c}(t))V^j(t)=0, \qquad V^i(0)=w^i.
\] 
However, even when $w$ is an eigenvector of $\Phi_{v}$, there is no guarantee that $V(t)$ remains an eigenvector of $\Phi_c$. Moreover, it is not clear if the eigenfunction $\lambda(v)$ remains constant along $\dot c(t)$.

In the next paragraphs we wish to examine some cases where the assumptions of this proposition are fulfilled simultaneously. Two such cases come immediately to mind: 

(1) Systems with $\Gamma(\lambda) =0$, for all eigenfunctions. Then, each $\lambda_c(t)$ is constant along each of the base  integral  curves of $\Gamma$. If we fix a base integral curve $c$ and if we can find a vector field $V(t)$ with $\nabla_cV(t)=0$, the conditions of the proposition are fulfilled for positive eigenfunctions.

(2) Systems where a  vector field $X\in\vectorfields\tau$ along $\tau$ exists with $\nabla X=0$. Then $\nabla_cX_c=0$ for each base integral curve $c$. If we now fix a base integral curve $c$, which has a  constant positive eigenvalue $\lambda_c$ along $c$, then the conditions of the proposition are again satisfied.

After some further preliminaries below, we give an application of each of the above cases in the next two sections.

Consider a \sode\  $\Gamma$. Let $D_\lambda$ be the eigendistribution of an eigenfuction $\lambda(v) \in \cinfty{TM}$ of $\Phi$. We will often simply write `eigen vector field' instead of the more formal  `eigen vector field along $\tau$'. Since we assume that $\Phi$ is diagonalizable, we know that an eigenbasis of $\vectorfields{\tau}$ exists.

\begin{proposition} Let $\lambda$ be an eigenfunction of $\Phi$. The following statements are equivalent:
\begin{enumerate}
\item $[\nabla\Phi,\Phi](D_\lambda) = 0$,
\item $\nabla\Phi(D_\lambda) \subset D_\lambda$,
\item  $\nabla D_\lambda \subset D_\lambda$.
\end{enumerate}
\end{proposition}

\begin{proof}  $(1)\Rightarrow (2)$. The bracket in  $[\nabla\Phi,\Phi]$ stands for the commutator. For a given $X \in D_\lambda$ it means that
\[
0 = (\nabla\Phi)(\Phi(X)) - \Phi((\nabla\Phi)(X)) = \lambda (\nabla\Phi)(X) - \Phi((\nabla\Phi)(X)),
\]  
or:  $(\nabla\Phi)(X)\in D_\lambda$.

$(2)\Rightarrow (3)$. Assume now that $(\nabla\Phi)(D_\lambda)\subset D_\lambda$. Since 
\[
(\nabla\Phi)(X)=\nabla(\Phi(X)) - \Phi(\nabla X) = \Gamma(\lambda) X + \lambda \nabla X - \Phi(\nabla X),  
\]
we may also conclude that  
\[
\lambda \nabla X - \Phi(\nabla X) \in D_\lambda.
\] If we take an eigenbasis $\{X_a,X_i\}$, where $X_a\in D_\lambda$, and  where $X_i$ lie in eigenspaces $D_{\lambda_i}$ (with then $\lambda_i\neq \lambda$) then we can decompose $\nabla X = A^aX_a + A^i X_i$. The previous relation then leads to 
\[
 (\lambda-\lambda_i) A^i X_i \in D_\lambda
\]
Since $\lambda\neq\lambda_i$, and since all $X_i$ are linearly independent we may conclude that all $A^i$ vanish. In conclusion, $\nabla X \in D_\lambda$ for all $X\in D_\lambda$.

$(3)\Rightarrow (1)$. When $\nabla X \in D_\lambda$ for all $X\in D_\lambda$, then 
\begin{eqnarray*}
[\nabla\Phi,\Phi](X) &=& (\nabla\Phi)(\Phi(X)) - \Phi((\nabla\Phi)(X)) = \lambda (\nabla\Phi)(X) - \Phi(\nabla(\Phi(X))) -\Phi(\Phi(\nabla X))\\
&=& \lambda (\nabla(\Phi(X)) - \Phi(\nabla X)) - \Phi(\nabla(\lambda X)) -\Phi(\lambda\nabla X)
\\
&=& \lambda (\nabla(\lambda X) - \Phi(\nabla X)) - \Phi(\Gamma(\lambda)+\lambda \nabla X) -\lambda^2\nabla X 
\\
&=& \lambda (\Gamma(\lambda) X + \lambda \nabla X - \lambda\nabla X) - \Phi(\Gamma(\lambda)X+\lambda \nabla X) -\lambda^2\nabla X
\\&=&0.
\end{eqnarray*}
\end{proof}

\begin{proposition} \label{nablaXzeroold}
Let $\lambda$ be an eigenfunction of $\Phi$ such that $D_\lambda$ is one-dimensional. Then the following statements are equivalent:
\begin{enumerate} \item  $[\nabla\Phi,\Phi](D_\lambda) = 0$,
\item $D_\lambda$ contains an eigen vector field $X$ for which $\nabla X=0$.
\end{enumerate}
\end{proposition}

\begin{proof}
In this case, from $[\nabla\Phi,\Phi](D_\lambda) = 0$ it follows that $\nabla Y = gY$, for all $Y\in D_\lambda$. Take now $X=\mu Y$, then $\nabla X = 0$ if and only if \[
\Gamma(\mu) + \mu g =0.
\]
Since such $\mu$ always exists, we may conclude the proof.

Conversely, when $D_\lambda$ contains an $X$ such that $\nabla X=0$, it follows for any other $Y=hX \in D_\lambda$ that $\nabla Y = \Gamma(h)X \in D_\lambda$.
\end{proof}

When we drop the condition on the dimension of $D_\lambda$, we still get:
\begin{proposition} \label{nablaXzeronew}
Let $\lambda$ be an eigenfunction of $\Phi$ such that  $[\nabla\Phi,\Phi](D_\lambda) = 0$. Then  $D_\lambda$ contains an eigen vector field $X$ for which $\nabla X=0$.
\end{proposition}

\begin{proof} Let $\{X_a\}$ be a basis of $D_\lambda$, where $a=1,\ldots, d ={\rm dim}(D_\lambda)$. We know that from $[\nabla\Phi,\Phi](D_\lambda) = 0$ it follows that $\nabla X_a = A_a^bX_b$. If we set $X=\mu^a X_a$, then $\nabla X = 0$ if and only if 
\[
\Gamma(\mu^a) + A_a^b \mu^b=0, \qquad a=1,\ldots, d.
\]
If we consider (not-natural) coordinates $(x^1,\ldots,x^{2n})$ on $TM$ that rectify $\Gamma$, i.e.\ $\Gamma=\fpd{}{x^1}$, then this equation is of the type
\[
(\mu^a)' + A_a^b \mu^b=0, \qquad a=1,\ldots, d,
\]
where prime denotes derivatives in $x^1$. When all other coordinates $(x^2,\ldots,x^{2n})$ are thought of as parameters this reduces to an initial value problem. 
\end{proof}

The condition $[\nabla\Phi, \Phi](D_\lambda) =0$ that we have encountered in the above propositions will automatically be satisfied when $[\nabla\Phi,\Phi] =0$ (regardless $D_\lambda$). This condition is quite familiar in the context of \sodes. For example, it is one of the conditions for a \sode\ to belong to ``Case II'' of the inverse problem of Lagrangian mechanics (see e.g.\ \cite{CPST}), and it is one of many conditions for a \sode\ to be ``separable'' (see e.g.\ \cite{sep}).

\section{Locally symmetric \sodes} \label{symspace}

For a Riemannian metric $g$ with geodesic (quadratic) spray $\Gamma$, the relation between the Jacobi endomorphism of $\Gamma$ and the curvature of the Levi-Civita connection is   $\Phi(Y)=R(Y,\T)\T$ (for any `basic' $Y\in \vectorfields M$). Moreover (see Proposition~\ref{Prop2})  for any quadratic spray,  $\nabla\T=0$. Then 
\begin{eqnarray*}
\nabla\Phi(Y) &=&  \nabla(\Phi(Y)) - \Phi(\nabla Y)\\
 &=&  \nabla(R(Y,\T)\T) - R(\nabla Y,\T)\T\\
&=&  (\nabla R)(Y,\T,\T) + R(\nabla Y,\T,\T) - R(\nabla Y,\T)\T\\
&=&  (\nabla R)(Y,\T,\T). 
\end{eqnarray*}

Since $R$ is a `basic' tensor field on $M$ (viewed here `along $\tau$'), the condition $(\nabla R)(Y,\T,\T)=0$ is cubic in the fibre coordinates, i.e.\ of the type $P^i_{jklm}(q){\dot q}^j{\dot q}^k{\dot q}^lY^m$. The coefficients $P^i_{jklm}(q)$ are actually those of the tensor field $DR$, where $D$ is the Levi-Civitaconnection of $g$ (as a coordinate calculation easily confirms). We may therefore conclude that a Riemannian space is {\em locally symmetric} (i.e.\  $DR=0$) if and only if $\nabla\Phi=0$.

The condition $\nabla\Phi=0$ can also be satisfied for a \sode\ (without it being a Riemannian geodesic spray).
The condition $[\Phi,\nabla\Phi](D_\lambda)=0$ is then, of course, also satisfied.

\begin{proposition} \label{nablaPhizero} Let $\lambda$ be an eigenfunction of $\Phi$ such that  $[\nabla\Phi,\Phi] (D_\lambda) =0$. Then the following statements are equivalent:
\begin{enumerate} \item  $\nabla\Phi(D_\lambda) = 0$,
\item The eigenfunction $\lambda\in \cinfty{TM}$ is a first integral of $\Gamma$, i.e.\ $\Gamma(\lambda)=0$.
\end{enumerate} 
\end{proposition}

\begin{proof}
We know that $\nabla Y \in D_\lambda$ for all $Y\in D_\lambda$. Thus
\[ (\nabla\Phi) (Y) = \nabla(\Phi(Y)) - \Phi(\nabla Y) = \Gamma(\lambda) Y + \lambda \nabla Y - \lambda \nabla Y = \Gamma (\lambda) Y
\]
\end{proof}

\begin{proposition} \label{Prop11}
Assume that $\Gamma$ is a \sode\ such that $\nabla\Phi=0$. Let $c$ be an integral curve of $\Gamma$. Then, $c(0)$ has conjugate points at $c\big(\frac{k\pi}{\sqrt{\lambda_c(0)}}\big)$, for each eigenfunction $\lambda$ with $\lambda_c(0)>0$.
\end{proposition}
\begin{proof}

We already know from the previous proposition that under these assumptions $\Gamma(\lambda) =0$, for each eigenfunction $\lambda$. Then, for each given base integral curve $c$ of $\Gamma$, the function $\lambda_c(t) = \lambda({\dot c}(t))$ remains constant along $c$ (i.e.\ $\lambda_c(t) =\lambda_c(0)$ for all $t$). When $\lambda_c(0)>0$, the first item of the conditions in Proposition~\ref{basicprop} is satisfied.

We now proceed with the construction of a $V(t)\in D_{\lambda_c}$ such that $\nabla_c V(t) =0$. Let $c(0)=m_0$ and ${\dot c}(0)=v$. Consider $w \in T_{m_0}M$ with $\Phi_v(w)=\lambda_0 w$. Consider its extension to a vector field $W(t)\in \vectorfields c$, as described before, with $\nabla_c W(t)=0$. Given that now $\nabla\Phi=0$, we also find that
\[
\nabla_c(\Phi_c(W(t))) = (\nabla_c\Phi_c) (W(t)) + \Phi_c(\nabla_c W(t)) = 0.
\] 
But, then $\Phi_c(W(t))$ is uniquely determined by its initial value, which is $\Phi_c(W(0)) = \lambda_0 w$. Since this is the same initial value of the vector field $\lambda_0 W(t)$, which also happens to satisfy $\nabla_c(\lambda_0 W(t))=0$, we may conclude that $\Phi_c(W(t)) = \lambda_0 W(t)$. So, $W(t)$ remains throughout a eigenvector with eigenvalue $\lambda_0=\lambda_c(t)$. This means that $W(t)\in D_{\lambda_c}$. All the conditions of Proposition~\ref{basicprop} are therefore satisfied.
\end{proof}

The above proposition extends a well-known result for locally symmetric spaces, when we specialize it to the case where $\Gamma$ is the geodesic spray of a Riemannian metric (see e.g. \cite{DoCarmo}, where it is an exercise).

\section{Conjugate points along relative equilibria} \label{sec6}

In this section, we assume that $\Gamma$ is a \sode\ with a connected symmetry Lie group $G$. We assume that the action $\psi: G \times M \to M$ of $G$ on $M$ is free and proper, from which it follows that $\pi^M: M \to M/G$ is a principal fibre bundle. The same then holds true for the induced tangent action $T\psi_g$ on $TM$, and for the bundle $\pi^{TM}: TM \to (TM)/G$. From the symmetry condition $TT\psi_g \circ \Gamma = \Gamma \circ T\psi_g$ it follows that there exists a {\em reduced} vector field $\gamma$ on $(TM)/G$, defined by
\[
\gamma \circ \pi^{TM}  = T\pi^{TM}\circ\Gamma.
\]

Let $\xi_M$ be the fundamental vector field on $M$, corresponding to $\xi \in\la$. Then, $\clift{\xi}_M$ is a fundamental vector field for the induced action on $TM$, by construction. Since we assume that the Lie group is connected, the invariance of $\Gamma$ under the $G$-action is equivalent with  $[\clift{\xi}_M,\Gamma]=0$, for all $\xi\in\la$.

\begin{proposition} Assume that $\Gamma$ is an invariant \sode. Let $\{Y_i\}$ be an invariant frame of vector fields on $M$. Let $\nabla Y_i = \lambda_i^j Y_j$ and $\Phi(Y_i) = {\phi}_i^j Y_j$, then $\lambda_i^j$ and $\phi_i^j$ are invariant functions on $TM$. 
\end{proposition}
\begin{proof}
We only need to show that $\clift{\xi}_M(\lambda_i^j)=0$ and $\clift{\xi}_M(\phi_i^j)=0$ when $[\clift{\xi}_M,\Gamma]=0$.

Let $Y$ be an invariant vector field on $M$, then $[\xi_M,Y]=0$ and therefore also $[\clift{\xi}_M,\clift Y] = [ {\xi}_M, Y   \clift ]=0$. Likewise, using the Jacobi identity, we can show that $[\clift{\xi}_M,[\Gamma,\vlift Y]]=0$. Together with the relation $\hlift Y =\onehalf (\clift Y - [\Gamma,\vlift Y])$ 
  this means that also $[\clift{\xi}_M,\hlift Y]=0$, from which we may conclude that the horizontal lift of a basic vector field is invariant.

From $[\Gamma,\vlift Y] = -\hlift Y + (\nabla Y\vlift)
 $ it now follows that also $(\nabla Y\vlift)$ is an invariant vector field on $TM$. This means that the coefficients $\lambda_i^j$ in $\nabla Y_i = \lambda_i^j Y_j$ are invariant functions on $TM$.

 With that in mind, we know that also $(\nabla Y_i\hlift) = \lambda_i^j \hlift{Y}_j$ is an invariant vector field on $TM$. 
Again by making use of the Jacobi identity we may show that $[\clift{\xi}_M,[\Gamma,\hlift Y]]=0$. Together with $[\Gamma,\hlift Y]= (\nabla Y\hlift) +
(\Phi(Y)\vlift)$, we find that $(\Phi(Y)\vlift)$ is an invariant vector field. This means that the coefficients $\phi^i_j$ in $\Phi( Y_i) = \phi_i^j Y_j$ are invariant functions on $TM$.
\end{proof}

\begin{proposition} \label{Prop13}
Assume that $\Gamma$ is an invariant \sode. The eigenfunctions $\lambda(v) \in \cinfty{TM}$ of $\Phi$ are invariant functions.
\end{proposition}
\begin{proof} Suppose the statement is not correct, i.e.\ there exist a $v\in TM$ -- and by continuity an open subset $U$, containing $v$ -- on which $\clift{\xi_M}(\lambda)\neq 0$. We show that this leads to a contradiction. 

Consider again an invariant frame $\{Y_i\}$ of vector fields on $M$ and the invariant coefficients $\phi^i_j$ of $\Phi$ with respect to this frame. The eigenfunctions are the solutions of the characteristic equation, 
\[
\lambda^n + a_{n-1}\lambda^{n-1} + \ldots +a_1\lambda+a_0=0,
\]
where the functions $a_i$ can essentially be constructed from taking sums and products of the $\phi_i^j$. They are therefore also invariant functions. We may therefore consider a derivative of the above by $\clift{\xi_M}$, for an arbitrary $\xi\in\la$. We find:
\[
(n\lambda^{n-1} + (n-1)a_{n-1}\lambda^{n-2} + \ldots +a_1)\,\clift{\xi_M}(\lambda)=0,
\]
Since we assume that $\clift{\xi_M}(\lambda)\neq 0$, we must have that $n\lambda^{n-1} + (n-1)a_{n-1}\lambda^{n-2} + \ldots +a_1=0$. But then, we may take another derivative by $\clift{\xi_M}$, which will lead again to a product of a factor of order $n-2$ in $\lambda$ and the factor $\clift{\xi_M}(\lambda)$. By taking sufficient derivatives we end up with $n!=0$, which is clearly wrong. We  reach the conclusion that $\clift{\xi_M}(\lambda)=0$.
 \end{proof}

A base integral curve $c$ is a {\em relative equilibrium\/} of an invariant \sode\
$\Gamma$ if it coincides with an integral curve of a fundamental
vector field of the action of $G$ on $M$.  This means that the base integral curve is of the type  $c(t) = \exp(t\xi)m_0 = \psi_{\exp(t\xi)}(m_0)$, with $c(0)=m_0$ and ${\dot c}(0)=v$ ($\exp$ denotes here the exponential map of the Lie group $G$). Moreover, ${\dot c}(t) = T\psi_{\exp(t\xi)}(v)$. A relative equilibrium therefore projects under $\pi^M$ on an equilibrium of the reduced vector field $\gamma$. We refer the reader to e.g.\ \cite{CMreleq} (and the reference therein) for more details on relative equilibria.

\begin{proposition} \label{Prop14} Assume that $\Gamma$ is an invariant \sode. The eigenfunctions $\lambda_c(t)$ along a relative equilibrium $c(t) = \psi_{\exp(t\xi)}(m_0)$ are constant. 
\end{proposition}

\begin{proof}
Since the eigenfunctions are invariant, we may conclude that $\lambda(w) = \lambda(T\psi_g(w))$, for all $w\in TM$.
This means that, along a relative equilibrium, the functions
\[
\lambda_c(t) = \lambda({\dot c} (t)) = \lambda(T\psi_{\exp(t\xi)}(v))=\lambda(v)
\] 
are constant in $t$. We may therefore simply write them as $\lambda_0$ (for a fixed relative equilibrium).
\end{proof}

We now return to the results in Propositions~\ref{basicprop} and \ref{nablaXzeronew}. Suppose $\Gamma$ is such that $[\nabla\Phi,\Phi] = 0$.

We conclude:
\begin{proposition} \label{jacobiRE} Let $\Gamma$ be a $G$-invariant \sode\  that satisfies $[\nabla\Phi,\Phi] = 0$. Let $c$ be a relative equilibrium, starting at $c(0)=m_0$, with ${\dot c}(0)=v$. Suppose that there is an eigenfunction $\lambda$ which is strictly positive in $(m_0,v)$. Then, the points $c\big(\frac{k\pi}{\sqrt\lambda}\big) = \exp \left(\frac{k\pi}{\sqrt\lambda}\xi\right) m_0$ are conjugate  to $m_0$, for all $k\in {\mathbb Z}$.
\end{proposition}
\begin{proof} Under the conditions in the statement, we know from Proposition~\ref{nablaXzeronew} that  there exist a $X\in D_\lambda$ with $\nabla X=0$. Its restriction $X_c$ to the relative equilibrium $c$ satisfies $\nabla_cX_c=0$. Since $\lambda(t)=\lambda_0>0$ is constant along $c$, the conditions for Proposition~\ref{basicprop} are satisfied.
\end{proof}

{\bf The special case $M=G$.} In view of the examples in the next section, we end this section by listing some coordinate expressions for the special case where the manifold $M$ is the Lie group $G$, and the action is given by left translations. In that case all right-invariant vector fields are infinitesimal generators of the action on $M=G$, and their complete lifts are the generators of the induces action on $TM=TG$. The manifold $(TM)/G$ can in this case be identified with the Lie algebra $\la$. 

We follow the notations of \cite{CMipLie, CMreleq}. Let $E_i$ be a basis of the Lie algebra $\la$, and ${\hat E}_i$ be the corresponding left-invariant  vector fields on $\la$. The invariant frame $\{Y_i={\hat E}_i\}$ forms a basis for $\vectorfields M$. We will use this frame to define coordinates on each tangent space $T_qM$. Let $v_q\in T_qQ$, then there exist coefficients $(w^i)\in \R^n$ (the  so-called  quasi-velocities) such that  $v_q=w^i{\hat E}_i(q)$. These quasi-velocities can be use to coordinatize $TM$.  The set  $\{\clift{\hat E}_i,\vlift{\hat E}_i\}$ forms a basis for $\vectorfields G$. A \sode\ can thus be expressed as
\[
\Gamma = w^i \clift{\hat E}_i + \gamma^i \vlift{\hat E}_i \in \vectorfields{TG},
\]
where $\gamma^i$ are invariant functions on $TG$ which can here be identified with functions on $\la$. The reduced vector field is then given by 
\[
\gamma = \gamma^i \fpd{}{w^i} \in \vectorfields{\la}.
\]
 In \cite{CMipLie}, the coefficients $\phi_j^l$ (as functions on $\g$) have been calculated, and are given by
\[
\phi_j^l  =
\onehalf\gamma^i \frac{\partial^2\gamma^l}{\partial w^i \partial
w^j} +\onehalf\gamma^i C^l_{ij}
-\oneq\fpd{\gamma^i}{w^j}\fpd{\gamma^l}{w^i} -\threeq
C^k_{ij}w^i\fpd{\gamma^l}{w^k} +\oneq w^iC^l_{ik}\fpd{\gamma^k}{w^j
}-\oneq w^m w^n C^k_{mj}C^l_{nk}.
\]
and
\[
\lambda^k_i = -\onehalf\left(\fpd{\gamma^k}{w^i}-w^jC^k_{ji}\right).
\]
The $C^k_{ij}$ herein are the structure constants of the Lie algebra, $[E_i,E_j]=C_{ij}^k E_k$.

The coefficients of the tensor field $\nabla\Phi$ in this frame, $(\nabla\Phi) ({\hat E}_i) = \psi^j_i {\hat E}_j$ are then given by
\[
\psi_i^j = \gamma(\phi_i^j) + \phi_i^k \lambda_k^j - \phi_k^j \lambda_i^k.
\]

\section{Examples and applications}

Most calculations in the example below were carried out with the help of Maple. 

\subsection{A worked-out example}

Let $(x,y)$ be coordinates on $\R^2$. We consider the system of second-order differential equations on $\R^2$ given by 
\[
\ddot x =- x,\qquad \ddot y =(\dot y +x\dot x)^3 -{\dot x}^2 +x^2 -1.
\]
Since this system is not quadratic in velocities, it falls out of the `classic' scope of Riemannian geometry. We use this example to clarify all the concepts we have introduced in the previous sections.

The corresponding \sode\ is 
\[
\Gamma= \dot x \fpd{}{x} + \dot y \fpd{}{y} - x \fpd{}{\dot x} + ((\dot y +x\dot x)^3 -{\dot x}^2 +x^2 -1)\fpd{}{\dot y}.
\] It has a symmetry vector field $\fpd{}{y}$, which corresponds to the fact that the additive group $\R$ is a symmetry group, with action $\psi: \R \times \R^2  \to\R^2, (\epsilon; x,y) \mapsto (x, y+\epsilon)$. We may therefore consider the reduced vector field $\gamma$. In this simple case, the coordinates on the reduced manifold $(T\R^2)/\R = (\R^2\times \R^2)/\R = \R^3$ can be given by $(x,w_1=\dot x, w_2=\dot y)$. In these coordinates $\gamma$ is 
\[
\gamma = w_1\fpd{}{x} -x \fpd{}{w_1} + ((w_2 +xw_1)^3 -w_1^2 +x^2 -1) \fpd{}{w_2}.
\] 
A relative equilibrium is an equilibrium of $\gamma$. Here there is only one, namely $(x=0,w_1=0,w_2=1)$. The corresponding full solution on the configuration manifold $\R^2$ that goes through $(0,0)$ is $c(t)=(x_0(t) = 0,y_0(t) = t)$. This base integral  lies on the $y$-axis. 

The Jacobi endomorphism is given by
\[
(\Phi^i_j) = \begin{bmatrix} 1 & 0 \\ \Phi^2_1 & 3(\dot y +x\dot x)( \frac{1}{4}(\dot y +x\dot x)^3-1) \end{bmatrix}.
\]
We have therefore eigenfunctions
\[
\lambda_1(x,y,\dot x,\dot y)=1,\qquad \lambda_2(x,y,\dot x,\dot y)=3(\dot y +x\dot x)( \frac{1}{4}(\dot y +x\dot x)^3-1),
\]
 which are clearly invariant under the action,  and whose eigenspaces are both one-dimensional. 

One easily verifies that $\nabla\Phi$ takes the form
\[
\begin{bmatrix} 0 & 0 \\ x(\nabla\Phi)^2_2 & (\nabla\Phi)^2_2 \end{bmatrix}.
\]
If we also take into account that  $\Phi^2_1 = \Phi^2_2 x -x$, it easily follows that $[\nabla\Phi,\Phi]=0$.

Along the relative equilibrium $c(t)$ we get  constant values for the eigenvalues: $\lambda_1=1>0$ and $\lambda_2 = -\frac{9}{4}<0$. Since the first is positive, all conditions of Proposition~\ref{jacobiRE} are satisfied and we find conjugate points along the relative equilibrium at parameter values $t=k\pi$. Consider $t=\pi$. The conjugate point is then $c(\pi) = (0,\pi)$.

For this example, it is not evident to write down a closed expression for arbitrary base integral curves. 
However, one easily verifies that the following family of curves, defines for each $s$ a base integral curve:
\[
c_s(t) = (x_s(t) = s\sin t, y_s(t) = t - \frac{1}{2}s^2 \sin^2 t).
\]
Each of these base integral curves start at $t=0$ in $(0,0)$ and go for $t=\pi$ through $(0,\pi)$. When $s=0$ we have the relative equilibrium. If we  put \[
J(t) = \frac{d}{ds} (c_s(t))\bigg\rvert_{s=0},
= \sin t \fpd{}{x}\mid_{c(t)},
\]
One easily verifies that $\Gamma^1_1 = \Gamma^2_1 = 0$, from which $\nabla\fpd{}{x} = 0$. Moreover $\Phi(\fpd{}{x}) = 1 \fpd{}{x}$. So,
\[
\nabla_c\nabla_c J(t) + \Phi(J(t)) = \nabla_c (\cos t \fpd{}{x}\bigg\rvert_{c(t)}) + \sin t \fpd{}{x}\bigg\rvert_{c(t)} = - \sin t \fpd{}{x}\bigg\rvert_{c(t)} +
\sin t \fpd{}{x}\bigg\rvert_{c(t)} = 0.\]
This shows that $J(t)$ is a Jacobi field. It is clear that this field vanishes at both $t=0$ and $t=\pi$.

Below is a plot of the solutions $c_s(t)$, for some $s\in [-1,1]$.
\[
\includegraphics[scale=0.3]{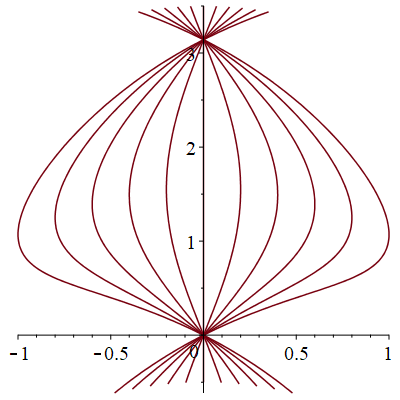}
\]

\subsection{Surfaces of revolution} 

We confirm here, by our methods, a result that appears in \cite{BCS} (who attribute it to H.\ Poincar\'e): The first conjugate point of a surface of revolution (under specific conditions, see later) is given by $\pi/\sqrt{K(0)}$, where $K(0)$ stands for the Gauss curvature at the initial point.

For a surface of revolution we may use the parametrization
\[
{\mathbf r}(\theta,\phi) = (f(\phi)\cos\theta,f(\phi)\sin\theta,g(\phi)), \qquad f>0,
\]
where $\phi$ is supposed to be arclength for the curve $\phi\mapsto (f(\phi),g(\phi))$ that is being rotated along the $z$-axis. An expression for the Gaussian curvature of this surface is then $K(\phi) = - \frac{f''(\phi)}{f(\phi)}$. 

We now consider its geodesics, as the solutions of the Euler-Lagrange equations of
\[
L = {\dot\phi}^2 + f^2(\phi) {\dot\theta}^2. 
\]
We also assume that geodesics are given by arclength, so that $L\equiv 1$. It is clear that this Lagrangian is invariant under the action $\theta \mapsto \theta +\epsilon$. The Euler-Lagrange equations and the condition on arclength are
\[
\ddot \phi = f(\phi)f'(\phi)\dot\theta,\qquad \ddot\theta  = -2 \frac{f'(\phi)}{f(\phi)}\dot\phi\dot\theta, \qquad {\dot\phi}^2 + f^2(\phi) {\dot\theta}^2 =1,
\]
or in invariant coordinates $( \phi, w_1=\dot\phi, w_2=\dot\theta)$,
\[
\dot\phi=w_1,\qquad {\dot w}_1 = f(\phi)f'(\phi)w_2,\qquad {\dot w}_2 = -2 \frac{f'(\phi)}{f(\phi)}w_1 w_2,\qquad w_1^2 + f^2(\phi) w_2^2=1.
\]
A relative equilibrium is therefore of the form $(\phi = \phi_0,w_1=0,w_2=\frac{1}{f(\phi_0)})$, where  $\phi_0$ is any constant that satisfies $f'(\phi_0)=0$.

As in \cite{BCS} we will assume in what follows that the `equator' $\phi_0=0$ is such that $f'(0)=0$. The corresponding curve that at $t=0$ goes through $(0,0)$ is then $(\phi(t) = 0,\theta(t) = \frac{1}{f(0)}t)$.

We now compute $\Phi,\nabla\Phi$ and $[\Phi,\nabla\Phi]$, for any choice of $(\theta,\phi,\dot\theta,\dot\phi)$. 

It turns out that, although $\nabla\Phi\neq 0$, the \sode\ is such that 
$[\Phi,\nabla\Phi] = 0$ (in any point, not only along relative equilibria). Since the full expressions of $\Phi$ and $\nabla\Phi$ contain many terms, we will only write down here their espressions in case of a specific example: For a torus we have $f=a+b\cos\phi$ and 
\[
\Phi = \frac{b\cos\phi}{f}\begin{bmatrix}  f^2{\dot\theta}^2 & - \dot\phi\dot\theta \\ -f^2 \dot\phi\dot\theta & {\dot\phi}^2 \end{bmatrix}, \qquad \nabla\Phi = \frac{ab\sin\phi \dot\theta}{f^2}\begin{bmatrix}  -f^2{\dot\theta}\dot\phi & - {\dot\phi}^2 \\ f^2{\dot\phi}^2  & -{\dot\theta}^2 \end{bmatrix}, 
\]
from which clearly also $[\Phi,\nabla\Phi] = 0$.

We may therefore be interested in the eigenvalues of $\Phi$. For  general $f(\phi)$, they take the form $\lambda=0$ (since $\Gamma$ is a spray) and 
\[
\lambda = -\frac{f''(\phi)}{f(\phi)} ({\dot\phi}^2 +f(\phi)^2 {\dot\theta}^2) = K(\phi), 
\]
where in the last equality we have used the fact that our curves are parametrized by arclength. Both eigenvalues have multiplicity 1. Along the equator, we conclude that the second eigenvalue is $K(0)$. Under the assumption that it is positive, we find conjugate points at times $\pi/\sqrt{K(0)}$. In the case of the torus (with $K(0) = \frac{b}{a+b}$) this result also  appears on page 190 of \cite{BCS2}.

\subsection{The free rigid body} 

The free rigid body is a Lagrangian system on $SO(3)$, that is invariant under the action of $SO(3)$ on itself. For that reason, the equations of motion are usually given in their reduced form on the Lie algebra $so(3)$. The so-called  Euler equations of
\[
l=\frac{1}{2} \left(I_1w_1^2+I_2w_2^2+I_3w_3^2\right)
\]
are the following equations: 
\begin{eqnarray*}
\dot{w}_1 & = & \frac{I_2-I_3}{I_1}w_2w_3, \\
\dot{w}_2 & = & \frac{I_3-I_1}{I_2}w_3w_1, \\
\dot{w}_3 & = & \frac{I_1-I_2}{I_3}w_1w_2.
\end{eqnarray*}
The right-hand sides of the above expressions define the vector field $\gamma$ on $\la=so(3)$.
  
A rotation matrix in Euler's coordinates is \[
R = \left( \begin{array}{ccc}
\cos\phi\cos\psi-\sin\phi\sin\psi\cos\theta & -\cos\phi\sin\psi-
\sin\phi\cos\psi\cos\theta & \sin\phi\sin\theta \\
\sin\phi\cos\psi+\cos\phi\sin\psi\cos\theta & -
\sin\phi\sin\psi+\cos\phi\cos\psi\cos\theta & -\cos\phi\sin\theta \\
\sin\psi\sin\theta & \cos\psi\sin\theta & \cos\theta
\end{array} \right).
\]
and the angular velocity is given by $\hat\omega= R^{-1} \dot R$, or by
\[ \left\{ \begin{array}{lcl}
w_1 & = & \dot{\theta}\cos\psi + \dot{\phi}\sin\theta\sin\psi , \\
w_2 & = & -\dot{\theta}\sin\psi + \dot{\phi}\sin\theta\cos\psi , \\
w_3 & = & \dot{\psi} + \dot{\phi}\cos\theta .
\end{array} \right. \]
The combination of the above, with the Euler equations gives in principle the Euler-Lagrange equations of the mechanical system on $SO(3)$ (i.e.\ before reduction).

From Euler's equations it is clear that the relative equilibria of the Euler equations are of the type $(w_1,w_2,w_3)=(\Omega_1,0,0)$, $(0,\Omega_2,0)$ or $(0,0,\Omega_3)$, where $\Omega_i$ can be any constant. 

We investigate here a relative equilibrium of the type $(0,0,\Omega)$, with $\Omega>0$ for convenience. Its base integral curve in $SO(3)$ through the unit is then $t\mapsto (\theta(t) =0  , \phi(t) =0, \psi(t) =\Omega t )$, or when written as a matrix:
\[
R(t) = \left( \begin{array}{ccc}
\cos(\Omega t) & -\sin(\Omega t)& 0 \\
\sin(\Omega t)  & \cos(\Omega t)  & 0 \\
0 & 0 & 1
\end{array} \right).
\]

Since we have an explicit expression of the reduced vector field $\gamma$, we may compute the coefficients of the  corresponding $\Phi$, $\nabla\Phi$ and $[\Phi,\nabla\Phi]$ by means of the expressions for $\phi^j_i$ and $\psi^j_i$ we gave in the previous section. 

It turns out that, after plugging in the relative equilibrium $(0,0,\Omega)$, the tensor $[\Phi,\nabla\Phi](0,0,\Omega)$ becomes 
\[\frac{\Omega^5}{2I_1^2I_2^2}\begin{bmatrix} 0 &   \frac{1}{I_2}(I_1 - I_2)^2(I_1+I_2 - I_3)^3 & 0\\ - \frac{1}{I_1}(I_1 - I_2)^2(I_1+I_2 - I_3)^3  & 0& 0\\0 & 0& 0\end{bmatrix}.
\]
There are therefore two cases where $[\Phi,\nabla\Phi](0,0,\Omega)$ vanishes: The case where $I_1=I_2$ (Euler top) and the case where $I_3=I_1+I_2$ (a flat rigid body). When any of these two cases is assumed, one finds that also $\nabla\Phi(0,0,\Omega) = 0$.

In the case where $I_1=I_2$ one may calculate that the Jacobi endomorphism has eigenvalue $\lambda_1=0$ and a double counted eigenvalue $\lambda_2=(\frac{I_3}{2I_1}\Omega)^2$. The eigendistribution of $\lambda_2$ (along the relative equilibrium) is spanned by the vector fields ${\hat E}_1$ and ${\hat E}_2$ (or: the vectors $E_1=(1,0,0)$ and $E_2=(0,1,0)$ when viewed as elements of the Lie algebra). 

The dynamical covariant derivative, when evaluated along the relative equilibrium $c(t)$ can be shown to have the behaviour: 
\[
\nabla_c {\hat E}_1(c(t)) = A {\hat E}_2(c(t)),\qquad \nabla_c {\hat E}_2(c(t)) = -A {\hat E}_1(c(t)),\qquad  \nabla_c {\hat E}_3(c(t)) = 0
\]
where $A= \Omega \frac{(2I_1-I_3)}{2I_1}$. For example, the vector field
\[
V(t) = \cos(At){\hat E}_1(c(t)) -\sin(At) {\hat E}_2(c(t)),
\] 
along $c(t)$ is a vector field in the eigendistribution of $\lambda_2$, that has the property that $\nabla_c V(t) =0$ (as predicted by Proposition~\ref{nablaXzeronew}).

We are therefore in the situation of Proposition~\ref{basicprop}, and we may conclude that, along the relative equilibrium $(0,0,\Omega)$, we have  a conjugate point at time $\frac{I_1}{I_3\Omega} 2\pi$.

This coincides with results in the literature (see e.g. \cite{BF,PS}). The conjugate time (for a general geodesic, not necessarily a relative equilibrium) is given in e.g.\ Theorem~1 of \cite{PS}. In their notations, our current situation $(0,0,\Omega) \in \la$ coincides with the conjugate momentum $p=(0,0,I_3\Omega) \in \la^*$. If we set $\bar p = \frac{p}{|p|} = (0,0,1)$, then ${\bar p}_3=1$. Theorem~1 then says that $\tau=\pi$, from which $t = \frac{2I_1\tau}{|p|} = \frac{I_1}{I_3\Omega} 2\pi$ is exactly as we found.

In the case where $I_3=I_1+I_2$ the Jacobi endomorphism has eigenvalue $\lambda=0$ and a double counted eigenvalue $\lambda=\Omega^2$. Along the relative equilibrium $(0,0,\Omega)$ we have therefore a conjugate point at time $\frac{\pi}{\Omega}$.

\subsection{The canonical connection on a Lie group} \label{seccancon} 

This is the (torsion free) linear connection on $G$ whose action on two left-invariant vector fields $X$ and $Y$ is given by 
\[
D_XY = \onehalf [X,Y]
\]
(see e.g.\ \cite{Thompson,Thompson2,CMipLie}).

In a basis of left-invariant vector fields we have $D_{{\hat E}_i} {\hat E}_j = \onehalf C_{ij}^k$. Due to the skew-symmetry in $C_{ij}^k$, the reduced equations are simply 
${\dot w}^i=C_{ij}^kw^iw^j=0$, and therefore $\gamma=0$. From this, it is clear that  $(w^i(t)=w^i_0)$, and that any element of the Lie algebra generates a relative equilibrium.

In \cite{Thompson2} it is shown that $R(X,Y)Z = \frac14 [Z,[X,Y]]$ (when $X,Y,Z$ are left-invariant vector fields) and that $D R = 0$. Similar to what we say in Section~\ref{symspace}, it follows from this that  $\nabla\Phi = 0$.

We use the classification of low-dimensional algebras given in \cite{Patera} to determine for which Lie algebras of dimension 3 the Jacobi endomorphism  $\Phi$ has positive eigenvalues. After some calculations one finds that only for the Lie algebras $A_{3,6}=e(2)$, $A_{3,8}=sl(2,\R)$ and $A_{3,9}=so(3)$ the Jacobi endomorphism  $\Phi$ has (at least one) positive eigenvalue. The last two algebras are semisimple, and the corresponding Killing form gives a bi-invariant metric whose Levi-Civita connection is the canonical connection.

The algebra $e(2)$, however, is not semisimple, and we will discuss below that there does not exist a Riemannian metric whose Levi-Civita connection is given by the canonical connection. This example, therefore, falls out of the scope of the methods from Riemannian geometry.

The bracket relations for $e(2)$ are $[E_1, E_2] = 0$, $[E_1, E_3] = -E_2$ and $[E_2, E_3] = E_1$. The corresponding connected Lie group is $SE(2)$, on which we may choose coordinates in such a way that an element  is given by
\[
\begin{bmatrix} \cos z & \sin z & x\\ -\sin z & \cos z & y\\ 0&0&1 \end{bmatrix}
\]
An element of $e(2)$ is then \[
\begin{bmatrix} 0 & w_3 & w_1\\ -w_3 & 0 & w_2\\ 0&0&0 \end{bmatrix}.
\]

The corresponding left invariant vector fields are 
\[
{\hat E}_1 = \cos z \fpd{}{x} + \sin z \fpd{}{y},\qquad {\hat E}_2 = -\sin z \fpd{}{x} + \cos z \fpd{}{y},\qquad {\hat E}_3 = \fpd{}{z}, 
\] 
from which 
\[
\left\{ \begin{array}{lll} \dot x &=& w_1 \cos z - w_2\sin z \\ \dot y &=& w_1\sin z + w_2 \cos z \\ \dot z &=& w_3   \end{array}\right.\quad
\Leftrightarrow \quad \left\{ \begin{array}{lll} w_1 &=& {\dot x} \cos z + {\dot y}\sin z \\ w_2 &=& -{\dot x} \sin z + {\dot y} \cos z \\ w_3 &=&\dot z     \end{array}\right. .
\]

The geodesic equations can be given in reduced form by 
\[
{\dot w}_1 = 0,\qquad {\dot w}_2 = 0,\qquad {\dot w}_3 = 0,
\]
or in the original coordinates as
\[
\ddot x = \dot y \dot z,\qquad \ddot y = -\dot x\dot z ,\qquad \ddot z = 0.
\]
The matrix $\Phi^i_j$ (in the coordinate frame) is here 
\[
\Phi = \frac{1}{4}\begin{bmatrix} {\dot z}^2 & 0 & 0 \\ 0 & {\dot z}^2 & 0\\ -{\dot x}{\dot z} & -{\dot y}{\dot z} & 0  \end{bmatrix}. 
\]
It is easy to see that there can not exist a Riemannian metric $g$ in this case. If that would be the case, then the geodesic equations would be Euler-Lagrange equations. That would mean that the \sode\ would have a solution for the so-called `inverse problem of Lagrangian mechanics'. The Riemannian metric $g$ and the \sode\ $\Gamma$ would then be related by the so-called Helmholtz conditions, see e.g. \cite{CPST,CMipLie,Thompson,Thompson2}. One of these conditions is that $g$ and $\Phi$ are related by
\[
g(\Phi(X), Y) = g(X,\Phi(Y)).
\]
In the current setting, the only non-vanishing conditions are
\[
-g_{11}{\dot x}  - g_{12}{\dot y}  - g_{13} {\dot z} = 0, \quad g_{12}{\dot x} - g_{22}{\dot y} - g_{23}{\dot z} = 0, \quad g_{13}{\dot z} + g_{11}{\dot x} + g_{12}{\dot y} = 0,\quad g_{23}{\dot z} + g_{12}{\dot x} + g_{22}{\dot y} = 0.
\]
It is clear that, in case of a Riemannian metric, $g_{ij}$ only depends on $x,y$ and $z$. Therefore, from the first equation we must get $g_{11}=g_{12}=g_{13}=0$, which can not happen for a non-degenerate metric. 

Thompson \cite{Thompson} has shown that the connection on $E(3)$ is variational, in the sense that there exist a local regular Lagrangian whose Euler-Lagrange equations are equivalent with the geodesic equations of the canonical connection. The  most general form of such a Lagrangian can be found in \cite{Thompson}, but one example is 
\[
L(x,y,z,\dot x,\dot y,\dot z) = \frac{1}{2\dot z}\left(({\dot y}^2-{\dot x}^2)\cos(z)+2{\dot x}{\dot y}\sin(z)\right)-\frac{1}{12}{\dot z}^4.
\] 
This confirms again that this example lies outside the realm of Riemann  (and even Finsler) geometry. Notwithstanding, our propositions  do provide an answer about conjugate points. 

One may calculate that the Jacobi endomorphism has the eigenvalues $\lambda = 0$ and $\lambda  = \frac{{\dot z}^2}{4}$ (double counted). If we consider the relative equilibrium $t \mapsto (x(t),y(t), z(t) = {\dot z}_0t )$ (through the unit $(x=0,y=0,z=0)$),
then the corresponding conjugate point is at $t =\frac{\pi}{\sqrt{\lambda}} = \frac{2\pi}{{\dot z}_0}$

We may also relate this result to the one we had found in Proposition~\ref{Prop6}. This proposition roughly says that  if the most general solution is given by $q(t,q_0,{\dot q}_0)$, then conjugate points are the points at instances $t_0$ such that the Jacobian 
\[
\fpd{q}{{\dot q}_0} (t_0,q_0,{\dot q}_0)
\]
is singular. 

In this case, we can actually solve the system of second-order differential equations in closed form. The most general solution that goes through the unit $(0,0,0)$ is given by 
\[
x(t) = \frac{{\dot x}_0}{{\dot z}_0}\sin({\dot z}_0t) - \frac{{\dot y}_0}{{\dot z}_0}\cos({\dot z}_0t)+\frac{{\dot y}_0}{{\dot z}_0},\qquad y(t) = \frac{{\dot x}_0}{{\dot z}_0}\cos({\dot z}_0t)  + \frac{{\dot y}_0}{{\dot z}_0}\sin({\dot z}_0t) - \frac{{\dot x}_0}{{\dot z}_0} ,\qquad z(t)= {\dot z}_0t, 
\]
where the integration constants are chosen in such a way that ${\dot x}(0)={\dot x}_0$, ${\dot y}(0) = {\dot y}_0$, ${\dot z}(0)={\dot z}_0$. With that, the above Jacobian becomes 
\[
\begin{bmatrix}  \frac{1}{{\dot z}_0}\sin({\dot z}_0t) & -\frac{1}{{\dot z}_0}\cos({\dot z}_0t) +\frac{1}{{\dot z}_0} &  \,\,\cdots\,\, \\ \frac{1}{{\dot z}_0}\cos({\dot z}_0t) -\frac{1}{{\dot z}_0} & \frac{1}{{\dot z}_0}\sin({\dot z}_0t) & \,\,\cdots\,\,\\
 0&0& t\end{bmatrix},
\]
with determinant $\frac{2t}{{\dot z}_0^2}\left(1 - \cos({\dot z}_0t)\right)$. Besides at $t=0$, this  matrix is singular at times $t= \frac{2k\pi}{{\dot z}_0}$, which agrees with what we had found before. 
The point through which all geodesics go at time $t=\frac{2\pi}{a}$ is  $( 0, 0 , 2\pi)$.

\section{Outlook}

In this paper we have investigated Jacobi fields and conjugate points in the context of \sodes. Of course, also regular Lagrangian systems and the geodesic equations of a Finsler metric fall in this category. Our next goal is to see whether the methods we have developped here can also be applied to provide new results in those areas. 

In particular, we wish to investigate in the context of Lagrangian systems and in the context of  Finsler geometry the notions of geodesics `dispersing' and `bunching together' \cite{DoCarmo,BCS2}, the notion of Jacobi  stability for \sodes\ \cite{Sabau} and a possible extension of Rauch's theorem \cite{BCS2} to Lagrangian systems.

\vspace*{5mm}

{\bf Acknowledgements.}

 %SH is a PhD fellow of the Research Fund of the University of Antwerp. 
We are indebted to Willy Sarlet for helpful comments on a first version of this manuscript. TM thanks the Research Foundation -- Flanders (FWO) for its support through Research Grant 1510818N and through the Excellence of Science Project G0H4518N. 

\bibliographystyle{plain}

\end{document}